\documentclass[11pt]{amsart}
\usepackage{amssymb}
\usepackage{amsfonts}
\usepackage{color}
\usepackage{xy}
\xyoption{all}
\usepackage{setspace}
\newtheorem{theorem}{Theorem}[section]
\newtheorem{proposition}[theorem]{Proposition}
\newtheorem{lemma}[theorem]{Lemma}
\newtheorem{corollary}[theorem]{Corollary}

\theoremstyle{remark}
\newtheorem{remark}[theorem]{Remark}

\renewenvironment{proof}{{\noindent\bf Proof.}}{\hfill $\Box$\par\vskip3mm}

\newcommand{\Hom}{{\rm Hom}}
\newcommand{\End}{{\rm End}}

\newcommand{\Aa}{\mathcal{A}}

\newcommand{\Cc}{\mathcal{C}}

\newcommand{\Gg}{\mathcal{G}}

\newcommand{\Ll}{\mathcal{L}}
\newcommand{\Mm}{\mathcal{M}}

\newcommand{\Tt}{\mathcal{T}}
\newcommand{\Ss}{\mathcal{S}}

\begin{document}
\title[Dickson Splitting for ${}_{C^*}\Mm$]{The Dickson Subcategory Splitting Conjecture for Pseudocompact Algebras}

\begin{abstract}
Let $A$ be a pseudocompact (or profinite) algebra, so $A=C^*$ where $C$ is a coalgebra. We show that the if the semiartinian part (the "Dickson" part) of every $A$-module $M$ splits off in $M$, then $A$ is semiartinian, giving thus a positive answer in the case of algebras arising as dual of coalgebras (pseudocompact algebras), to a well known conjecture of Faith.
\end{abstract}

\author[Iovanov, Nastasescu, Torrecillas]{Miodrag Cristian Iovanov$^1$, Constantin Nastasescu$^2$, Blas Torrecillas}
\thanks{$^1$ The author was partially supported by the contract nr. 24/28.09.07 with UEFISCU "Groups, quantum groups, corings and representation theory" of CNCIS, PN II (ID\_1002)}
\thanks{$^2$ The author was jointly supported by Grant 434/1.10.2007 of CNCSIS, PN II (ID\_1005) and by a Ramon y Cajal'project of the University of Almeria.}
\date{}
\maketitle

\section*{Introduction}

Let $A$ be a ring and $T$ be a torsion preradical. Then $A$ is said to have splitting property provided that $T(M)$, the torsion submodule of $M$, is a direct summand of $M$ for every $A$-module $M$. More generally, if $\Cc$ is a Grothendieck category and $\Aa$ is a subcategory of $\Cc$, then $\Aa$ is called closed if it is closed under subobjects, quotient objects and direct sums. To every such subcategory we can associate a preradical $t$ (also called torsion functor) by putting $t(M)=$the sum of all subobjects of $M$ that belong to $\Aa$. We say that $\Cc$ has the splitting property with respect to $\Aa$ if it has the splitting property with respect to $t$, that is, if $t(M)$ is a direct summand of $M$ for all $M$. The subcategory $\Aa$ is called localizing if $\Aa$ is closed and also closed under extensions. In the case of category of $A$ modules, the splitting property with respect to some closed subcategory is a classical problem which has been considered by many authors. In particular, the question of when the (classical) torsion part of an $A$ module splits off is a well known problem. J. Rotman has shown in \cite{Rot} that for a commutative domain all modules split if and only if $A$ is a field. I. Kaplansky proved in \cite{K1}, \cite{K2} that for a commutative integral domain $A$ the torsion part of every finitely generated $M$ module splits in $M$ if and only if $A$ is a Pr$\rm\ddot{u}$fer domain. While complete results have been obtained in the commutative case, the characterization of the noncommutative rings $A$ for which (a certain) torsion splits in every $A$ module (or in every finitely generated module) is still an open problem. \\
Another well studied problem is that of the singular splitting. Given a ring $A$ and an $A$-module $M$, denote $Z(M)=\{x\in M\mid Ann(x){\rm \,is\,an\,essential\,ideal\,of\,}A\}$. Then a module is called singular if $M=Z(M)$ and nonsingular if $Z(M)=0$. Then, a ring $A$ is said to have the (finitely generated) singular splitting property if $Z(M)$ splits in $M$ for all (finitely generated) modules $M$. A thorough study and complete results on this problem was carried out in the work of M.L. Teply; see (also) \cite{G}, \cite{FK}, \cite{FT}, \cite{T1}, \cite{T2} (for a detailed history on the singular splitting), \cite{T3}. \\
Given a ring $A$, the smallest closed subcategory of the category of left $A$-modules ${}_A{\mathcal M}$, containing all the simple $A$-modules, is obviously the category of semisimple $A$-modules. Then one can always consider another more suitable "canonical" subcategory , namely include all simple $A$-modules and consider the smallest localizing subcategory of ${}_A{\mathcal M}$ that contains all these simple modules (recall that a subcategory is called localizing if it is a closed subcategory and if it is closed under extensions). This category is called the Dickson subcategory of ${}_A{\mathcal M}$, and it is well known that it consists of all semiartinian modules \cite{D}. More generally, this construction can be done in any Grothendieck category ${\mathcal C}$. Thus one can consider the splitting with respect to this Dickson subcategory; if a ring has this splitting property, we will say it has the Dickson splitting property. A remarkable conjecture in ring theory asks the question whether if a ring $A$ has this splitting property, then does it necessarily follow that $A$ is semiartinian? Obviously the converse is trivially true. The answer to this question in general has turned out to be negative. In this respect, an example of J.H. Cozzens in \cite{C} shows that there is a ring $R$ (a ring of differential polynomials) that is not semisimple and, among other, has the property that any simple right $R$-module is injective (in fact it has a unique simple right module up to isomorphism) and is noetherian on both sides. Then, if $A=R^{op}$ then the Dickson subcategory of ${}_A{\mathcal M}$ coincides with the that of semisimple $A$-modules and the (left) Dickson splitting property obviously holds since then all semisimple modules are injective ($A$ is left noetherian). However, this ring is not semisimple and thus not (left) semiartinian.\\
Motivated by these facts, in this paper we consider the case when the ring $A$ is a pseudocompact algebra: an algebra which is a topological algebra with a basis of neighbourhoods of 0 consisting of ideals of $A$ of finite codimension and which is Hausdorff separate and complete. Equivalently, such an algebra is an inverse (pro-) limit of finite dimensional algebras, and thus they are also called profinite algebras and their theory extends and generalizes, in part, the theory of finite dimensional algebras. This class of algebras is one very intensely studied in the last 20 to 30 years; they are in fact the algebras that arise as a dual (convolution) algebra of a coalgebra $C$, and the theory of the representations of such algebras is well understood through the theory of corepresentations (comodules) of coalgebras. In fact, if $A=C^*$ for a coalgebra $C$, the category of pseudocompact left $A$-modules is dual to that of the left $C$-comodules; see \cite{DNR}, Chapter 1. The main result of the paper shows that the conjecture mentioned above holds for this class of algebras. The particular question of whether this holds for algebras that are dual of a coalgebra $C$ was also mentioned in \cite{NT}. As a direct and easy consequence, we re-obtain the main result from \cite{I1} and \cite{NT} stating that if a coalgebra $C$ has the property that the rational submodule of every left $C^*$-module $M$ splits off in $M$, then $C$ must be finite dimensional.\\
We extensively use the notations and language of \cite{DNR}; for general results on coalgebras and comodules, we also refer to the well known classical textbooks \cite{A} and \cite{S}. We first give some general results about a coalgebra $C$ for which the Dickson splitting property for $C^*$ holds. We show that such a coalgebra must be almost connected (must have finite dimensional coradical) and also that if a coalgebra $C$ has this property then any subcoalgebra $D$ of $C$ has this Dickson splitting property for $D^*$. In some special cases, such as when the Jacobson radical of $C^*$ is left finitely generated (an in particular when $C^*$ is left noetherian or when $C$ is an artinian right $C^*$-module) or when $C^*$ is a domain, then the Dickson splitting property implies that the coradical filtration of $C$ is finite, and consequently, in this case, $C^*$ is semiartinian, and moreover, it has finite Loewy length. For the general case, we first show the Dickson splitting property for $C^*$ implies $C^*$-semiartinian for colocal coalgebras (i.e. when $C^*$ is a local ring), and then treat the general case by using standard localization techniques, some general and some specific to coalgebras. The main proofs will include some extensions and generalizations of an old idea from abelian group theory and will make use of general facts from module theory but also of a number of techniques specific to coalgebra (and corepresentation) theory.


\section{General results}

For a vector space $V$ and a subspace $W\leq V$ denote by $W^{\perp}=\{f\in V^*\mid f(x)=0,\,\forall\,x\in W\}$ and for a subspace $X\leq V^*$ denote by $X^\perp=\{x\in V\mid f(x)=0,\,\forall\,f\in X\}$. Recall from \cite{DNR} that for any subspace $X$ of $C$, $X$ is a left coideal (or right coideal, or respectively subcoalgebra) if and only if $X^\perp$ is a left ideal (or right ideal, or respectively two-sided ideal) of $C^*$. Similarly, if $I$ a left (or right, or two-sided) ideal of $C^*$, then $I^\perp$ is a right coideal (or left coideal, or subcoalgebra) of $C$. Moreover, if $X<C$ is a left (or right $C$) subcomodule (coideal) of $C$ then there is an isomorphism of left $C^*$-modules $(C/X)^*\simeq X^\perp$. We consider the finite topology on $C^*$; recall that $(X^{\perp})^{\perp}=X$ for any subspace of $C$ and also for a subspace $X$ of $C^*$, $(X^\perp)^{\perp}=\overline{X}$, the closure of $X$. Consequently, $(X^\perp)^\perp=X$ if and only if $X$ is closed. \\
Throughout, $C$ will be a coalgebra and $\varepsilon$ will be the counit of the coalgebra. Also $J=J(A)$ will denote the Jacobson radical of $A=C^*$; then one has that $J=C_0^\perp$ and also $J^\perp=C_0$. Generally, for a left $A$-module $M$, $J(M)$ will denote the Jacobson radical of $M$; 


Let $\Ss$ be a system of representatives for the simple left comodules and for $S\in\Ss$, let $C_S$ be the coalgebra associated to $S$; then $C_S=\sum\limits_{T<C,T{\rm simple},T\simeq S}T$ is a finite dimensional coalgebra and $C_0=\bigoplus\limits_{S\in\Ss}C_S$ (see \cite{DNR}, proposition 2.5.3 and Chapter 3.1). Denote by $\mathcal T$ the torsion preradical associated to the Dickson subcategory of ${}_A\Mm$. Note that $A/J=C^*/C_0^\perp\simeq C_0^*=\prod\limits_{S\in\Ss}A_S$ where $A_S=C_S^*$ as left $A$-modules.

\begin{proposition}\label{p.0}
With the above notations, $\Sigma=\sum\limits_{S\in\Ss}A_S\subseteq A/J=\prod\limits_{S\in\Ss}A_S$ is the socle of $A/J$ and moreover, $\Sigma={\mathcal T}(A/J)$.
\end{proposition}
\begin{proof}
For $x=(x_S)_{S\in\Ss}\in \Pi=\prod\limits_{S\in\Ss}A_S$ denote $supp(x)=\{S\in\Ss \mid x_S\neq 0\}$. Obviously, $\Sigma$ is a semisimple module. It is enough to see that $\Pi/\Sigma$ contains no simple submodules. Assume by contradiction that $(Ax+\Sigma)/\Sigma$ is a simple (left) submodule of $\Pi/\Sigma$; then obviously $supp(x)$ is infinite ($x\notin \Sigma$) and write $supp(x)=I\sqcup J$ a disjoint union with infinite $I$ and $J$. Take $X$ such that $X\oplus C_0=C$ and let $e_I$ be defined as $\varepsilon$ on $\bigoplus\limits_{S\in I}C_S$ and $0$ on $\bigoplus\limits_{S\in J}C_S\oplus X$; put $x_I=e_I\cdot x$. Since $A_S\simeq C_S^*$ as $A$-bimodules, for $x_S\in A_S$, by using the left $C$-comodule structure of $C_S$, since $\Delta:C_S\rightarrow C_S\otimes C_S$, $A_S=C_S^*$ has a right comultiplication $A_S\rightarrow A_S\otimes C_S\subseteq A_S\otimes C$. Then it is easy to see that $e_I\cdot x_S=0$ if $S\notin I$ and $e_I\cdot x_S=x_S$ when $S\in I$. This shows that $supp(x_I)=I$. We have an inclusion of modules
$$ \Sigma\subsetneq \Sigma+Ax_I\subsetneq \Sigma+Ax$$
The strict inclusions hold since $x_I\notin \Sigma$ ($supp(x_I)=I$ is infinite) and $x\notin \Sigma+Ax_I$ (otherwise, $x=\sigma+ax_I$, $\sigma \in \Sigma$, $a\in A$ so then $supp(x)\subseteq supp(\sigma)\cup supp(x_I)$ following that $J=supp(x)\setminus I\subseteq (supp(\sigma)\cup I)\setminus I \subseteq supp(\sigma)$ which is finite, a contradiction). This shows that $(\Sigma+Ax)/\Sigma$ is not simple and the proof is finished.
\end{proof}

\begin{corollary}\label{p.1}
If $C$ is a coalgebra such that ${\mathcal T}(M)$ is a direct summand of $M$ for every cyclic $A$-module $M$, then $\Ss$ is a finite set and $C_0$ is finite dimensional.
\end{corollary}
\begin{proof}
Since $A/J$ is cyclic and $\Sigma={\mathcal T}(A/J)$, $\Sigma=\bigoplus\limits_{S\in \Ss}A_S$ is a direct summand of $A/J$ and thus it is itself cyclic. This shows that $\Ss$ must be finite, and therefore $C_0=\bigoplus\limits_{S\in \Ss}C_S$ is finite dimensional. 
\end{proof}

\begin{proposition}\label{p.2}
If $C$ is a coalgebra such that $C^*$ has the Dickson splitting property for left modules and let $D$ be a subcoalgebra of $C$. Then $D^*$ has the Dickson splitting property for left modules too.
\end{proposition}
\begin{proof}
Let $M$ be a left $D^*$-module. Let $I=D^\perp$, so we have an exact sequence
$$0\rightarrow I\rightarrow C^*\rightarrow D^*\rightarrow 0$$
Then $M$ is a left $C^*$-module through the restriction morphism $C^*\rightarrow D^*$ such that $I\subseteq Ann(M)$. Then there is a decomposition $M=\Sigma\oplus X$ where $\Sigma$ is the semiartinian part of $M$ as a $C^*$-module, and $X$ is a $C^*$-submodule of $M$. But $IM=0$ and therefore $I$ also annihilates both $\Sigma$ and $X$, hence $\Sigma$ and $X$ are also $D^*$-modules. Now note that if $S$ is a $D^*$-module, then the lattice of $C^*$-submodules of $S$ coincides to that of the $D^*$-submodules since $S$ is annihilated by $I$. This shows that $S$ is semiartinian (or has no semiartinian submodule) as $D^*$-module if and only if it is semiartinian as a $C^*$-module (respectively has no semiartinian submodule). Therefore, $\Sigma$ is semiartinian as $D^*$-module (as it is a semiartinian $C^*$-module) and $X$ contains no simple $D^*$-submodule (since $X$ contains no simple $C^*$-submodules), hence $\Sigma$ is the semiartinian part of $M$ also as a $D^*$-module and it splits in $M$.
\end{proof}

\subsection{Some general module facts}

We dedicate a short study for a general property of modules which is obtained with a "localization procedure", that will be used towards our main result. Although this follows in a much more general setting in Grothendieck categories (\cite{CICN}), we leave the general case treated there aside, and present a short adapted version here. We will often use the following 

\begin{remark}\label{r.J}
If $A/J$ is a semisimple algebra and $M$ is a left $A$-module, then $M$ is semisimple if and only if $J\cdot M=0$. Moreover, then for any left $A$-module $M$ we have $J(M)=JM$. Indeed, $M/J(M)\subseteq \prod\limits_{X<M,X{\,\rm maximal}}M/X$ which is canceled by $J$ and thus it is semisimple; therefore $J(M/J(M))=0$ i.e. $JM\subseteq J(M)$. Conversely, since $M/JM$ is semisimple, $JM$ is an intersection of maximal submodules of $M$, so $J(M)\subseteq JM$.
\end{remark}

To the end of this section, let $A$ be a ring, $e$ an idempotent of $A$. The functor $\Tt_e=eA\otimes_A-:{}_A\Mm\longrightarrow{}_{eAe}\Mm$ is exact and has $\Gg_e=\Hom_{eAe}(eA,-)$ as a right adjoint; in fact, $eA\otimes_AM\simeq eM$ as left $eAe$-modules for any left $A$-module $M$. Recall that for $N\in{}_{eAe}\Mm$, the left $A$-module structure on $\Hom_{eAe}(eA,N)$ is given by $(a\cdot f)(x)=f(xa)$, for $a\in A$, $x\in eA$, $f\in\Hom_{eAe}(eA,N)$. Let $\psi_{e,M}:M\rightarrow \Hom_{eAe}(eA,eM)$ be the canonical morphism (the unit of this adjunction); it is given by $\psi_{e,M}(m)(ea)=eam$ for $m\in M$, $a\in A$. The following proposition actually says that the counit of this adjunction is an isomorphism; the proof is a straightforward computation and is omitted.

\begin{proposition}\label{mod.1}
Let $N\in{}_{eAe}\Mm$ and for $n\in N$ let $\chi_n\in\Hom_{eAe}(eA,N)$ be such that $\chi_n(ea)=eae\cdot n$. Then the application
$$N\ni n\longmapsto e\cdot \chi_n\in e\cdot\Hom_{eAe}(eA,N)$$
is an isomorphism of left $eAe$-modules.
\end{proposition}

\begin{proposition}\label{mod.2}
Let $N$ be a left $eAe$-module and $X$ an $A$-submodule of $\Hom_{eAe}(eA,N)$. Denote $X(e)=\{f(e)\mid f\in X\}$. Then $X(e)$ is a submodule of $N$, and $X(e)\neq 0$ and $e\cdot X\neq 0$, provided that $X\neq 0$.
\end{proposition}
\begin{proof}
Let $x=f(e)\in X(e)$, $f\in X$ and $a\in A$. Then $eaex=eae\cdot f(e)=f(eae)=(ae\cdot f)(e)\in X(e)$ since $ae\cdot f\in X$. Moreover, if $f\neq 0$, then $f(ea)\neq 0$ for some $a$ and as above $0\neq f(ea)=(a\cdot f)(e)\in X(e)$; also $(e\cdot (a\cdot f))(e)=(a\cdot f)(e)=f(ea)\neq 0$, so $0\neq e\cdot (a\cdot f)\in e\cdot X$.
\end{proof}

\begin{proposition}\label{mod.3}
If $N\in{}_{eAe}\Mm$ has essential socle, then $\Gg_e(N)$ has essential socle too (as an $A$-module).
\end{proposition}
\begin{proof}
Let $0\neq H<\Gg_e(N)$ be a submodule of $\Gg_e(N)$ (assume $\Gg_e(N)\neq 0$). Then $H(e)\neq 0$ by Proposition \ref{mod.2} and $H(e)\cap s(N)\neq 0$. Let $\Sigma_0$ be a simple $eAe$-submodule of $H(e)$. We have an exact sequence 
$$0\rightarrow S\rightarrow \Gg_e(\Sigma_0)\rightarrow \prod\limits_{0\neq X<\Gg_e(\Sigma_0)}\frac{\Gg_e(\Sigma_0)}{X}$$
where $S=\bigcap\limits_{0\neq X<\Gg_e(\Sigma_0)}X$. Since $\Tt_e$ is exact, we have $e(\Gg_e(\Sigma_0)/X)\simeq e\Gg_e(\Sigma_0)/eX=0$ because $e\Gg_e(\Sigma_0)\simeq \Sigma_0$ by Proposition \ref{mod.1}, $\Sigma_0$ is simple and $eX\neq 0$ by Proposition \ref{mod.2}. Then, it easily follows that 
$$e\cdot (\prod\limits_{0\neq X<\Gg_e(\Sigma_0)}\frac{\Gg_e(\Sigma_0)}{X})=\Tt_e(\prod\limits_{0\neq X<\Gg_e(\Sigma_0)}\frac{\Gg_e(\Sigma_0)}{X})=0$$
and then by the above exact sequence and the exactness of $\Tt_e$ we get $S\neq 0$ (otherwise, if $S=0$, it follows that $\Gg_e(\Sigma_0)=0$, so $\Sigma_0\simeq e\Gg_e(\Sigma_0)=0$, a contradiction). Also, $S$ is simple by construction. Let $0\neq x\in \Sigma_0\subseteq H(e)$, $x=h(e)$, $h\in H$. Since there is a monomorphism $0\rightarrow \Gg_e(\Sigma_0)\rightarrow \Gg_e(N)$ and $0\neq eh$ has image contained in $\Sigma_0$ ($eh(ea)=h(eae)=eae\cdot h(e)\in \Sigma_0$ and $eh(e)=h(e)\neq 0$), we observe that $S\subseteq A\cdot eh$, by the construction of $S$. But $Aeh\subseteq H$, and thus $H$ contains the simple $A$-submodule $S$.
\end{proof}

\begin{theorem}\label{1.semiartinian}
Assume $A=\bigoplus\limits_{i\in F}E_i$ as left $A$-modules, and let $E_i=Ae_i$ with orthogonal idempotents $e_i$ with $\sum\limits_{i\in F}e_i=1$. Let $M$ be an $A$-module such that $e_iM=\Tt_{e_i}(M)$ is a semiartinian $e_iAe_i$-module for all $i\in F$. Then $M$ is semiartinian too. Consequently, if $e_iAe_i$ is a left semiartinian ring for all $i\in F$, then $A$ is left semiartinian too.
\end{theorem}
\begin{proof}
Obviously there always exist such idempotents $e_i$ and $F$ is finite. It is enough to show any such $M$ contains a simple submodule; if this holds, for any submodule $N$ of $M$, if $e_iM$ is $e_iAe_i$-semiartinian, then  $e_i(M/N)\simeq e_iM/e_iN$ is semiartinian over $e_iAe_i$ for all $i\in F$ and thus $M/N$ contains a simple submodule. \\
Let $M\rightarrow \overline{M}=\bigoplus\limits_{i\in F}\Gg_{e_i}(e_iM)$ be the canonical morphism, $m\mapsto (\psi_{e_i,M}(m))_{i\in F}$. This is obviously injective: $\psi_{e_i,M}(m)=0$ for all $i\in F$ implies $e_im=0,\,\forall i\in F$ so $m=1\cdot m=\sum\limits_{i\in F}e_i\cdot m=0$. By Proposition \ref{mod.3} $\overline{M}$ has essential socle, and so $s(\overline{M})\cap M\neq 0$ (provided $M\neq 0$) and this ends the proof. The last statement follows for $M={}_AA$.
\end{proof}

\begin{remark}
It not difficult to see that if $M$ is semiartinian over $A$ then $eM$ is semiartinian for any idempotent $e$; this is also a consequence of the more general results of \cite{CICN} or can be again seen directly.
\end{remark}

\section{The domain case}

We show that if $C$ is a coalgebra such that $C^*$ is a (local) domain and $C^*$ has the Dickson splitting property (that is, the semiartinian part of every left $C^*$-module splits off), then $C$ has finite Loewy length (in fact in this case $C^*$ will be a division algebra). We will again make use of the fact that if $X$ is a left subcomodule of $C$ then $X^\perp$ is a left ideal of $C^*$ and there is an isomorphism of left $C^*$-modules $(C/X)^*\simeq X^\perp$. 


\begin{remark}
For an $f\in C^*$ denote by $\overline{f}:C\rightarrow C$ the morphism of left $C$-comodules defined by $\overline{f}(c)=c_1f(c_2)$. Then, the maps $C^*\ni f\longmapsto \overline{f}\in{\End({}^CC)}$ and $\End({}^CC)\ni\alpha\longmapsto \varepsilon\circ f\in C^*$ are inverse isomorphisms of $k$-algebras (for example by \cite{DNR}, Proposition 3.1.8 (i)).
\end{remark}

\begin{lemma}\label{Dom.p.domain}
Let $C$ be a coalgebra. Then $C^*$ is a domain if and only if any nonzero morphism of left (or right) $C$-comodules $\alpha:C\rightarrow C$ is surjective. Moreover, if this holds, $C^*$ is local. 
\end{lemma}
\begin{proof}
Since $J=C_0^\perp$, we have $C^*/J\simeq C_0^*$ and therefore $C^*$ is local if and only if $C_0^*$ is a simple left $C^*$-module, equivalently $C_0$ is a simple comodule (left or right). \\
Assume first $C^*$ is a domain, so is $\End({}^CC)\simeq C^*$ is a domain. If $C_0$ is not simple, then there is a direct sum decomposition $C_0=S\oplus T$, where $S$ and $T$ are semisimple left $C$-comodules that are nonzero. In this case, if $E(S)$ and $E(T)$ are injective envelopes of $S$ and $T$ respectively that are contained in $C$, we get that $C=E(S)\oplus E(T)$ because $C_0$ is essential in $C$ and $C$ is injective (as left $C$-comodules). Then defining $\alpha,\beta:C\rightarrow C$ by 
$$
\alpha=\{
\begin{array}{cc}
	{\rm Id\;\;\;on\,}E(S)\\
	0{\;\;\;\;\;\rm on\,}E(T)
\end{array}
$$
and 
$$
\beta=\{
\begin{array}{cc}
	0{\;\;\;\;\;\rm on\,}E(S)\\
	{\rm Id\;\;\;on\,}E(T)
	\end{array}
$$
we obviously have $\alpha,\beta\neq 0$ and $\alpha\circ\beta=0$, showing that $\End({}^CC)$ cannot be a domain in this case. Therefore $C^*$ is local. Next, if some nonzero morphism of left $C$-comodules $\alpha:C\rightarrow C$ is not surjective, then ${\rm Im}(\alpha)\neq C$ so $C/{\rm Im}(\alpha)\neq 0$. Therefore there is a simple left subcomodule of $C/{\rm Im}(\alpha)$, that is, a monomorphism $\eta:C_0\rightarrow C/{\rm Im}(\alpha)$ (because $C_0$ is the only type of simple $C$-comodule). Then the inclusion $i:C_0\rightarrow C$ extends to $\beta_0:C/{\rm Im}(\alpha)\rightarrow C$, such that $\beta_0\eta=i$. If $\beta:C\rightarrow C$ is the composition of $\beta_0$ with the canonical projection $C\rightarrow C/{\rm Im}(\alpha)$ then obviously $\beta\vert_{{\rm Im}(\alpha)}=0$, so $\beta\circ\alpha=0$. But $\beta\neq 0$ because $\beta_0\neq 0$ since $\beta_0$ extends a monomorphism ($i$), and also $\alpha\neq 0$, yielding a contradiction to $C^*$-domain.\\
The converse implication is obvious: any composition of surjective morphisms is surjective, thus nonzero.
\end{proof}

\begin{lemma}\label{Dom.l.2}
If $C$ is colocal (equivalently $C^*$ is local), then for any left subcomodule $Y$ of $C$, we have that $Y^*$ is a local, cyclic and indecomposable left $C^*$-module.
\end{lemma}
\begin{proof}
The epimorphism of left $C^*$-modules $p:C^*\rightarrow Y^*\rightarrow 0$ induces an epimorphism $C^*/J\rightarrow Y^*/J(Y^*)$ which must be an isomorphism since $C^*/J$ is simple and $Y^*/J(Y^*)$ is nonzero because $Y^*$ is finitely generated (cyclic) over $C^*$. This shows that $Y^*$ is local. If $Y^*=A\oplus B$ with $A\neq 0$, $B\neq 0$ then $A$ and $B$ would be cyclic too, and therefore we would get $J(A)\neq A$ and $J(B)\neq B$. But $A/J(A)\oplus B/J(B)=Y^*/J(Y^*)$ is simple, and therefore showing that $A/J(A)=0$ or $B/J(B)=0$, a contradiction.
\end{proof}

\begin{lemma}\label{Dom.l.3}
If $M$ is a left $C$-comodule, then $\bigcap\limits_{n}(J^n\cdot M^*)=0$.
\end{lemma}
\begin{proof}
Let $f\in\bigcap\limits_nJ^n\cdot M^*$. Pick $x\in M$, and write $\rho(x)=\sum\limits_sc_s\otimes x_s\in C\otimes X$ the comultiplication of $x$. Since $\bigcup\limits_{n}C_n=C$, $c_s\in C_n,\,\forall s$ for some $n$. 
Also, as $f\in J^{n+1}\cdot M^*$, $f=\sum\limits_jf_jm_j^*$, $f_j\in J^{n+1},\,m_j^*\in M^*$. Since $(J^{n+1})^\perp=C_n$ we get $f_j(c_s)=0$ for all $j$ and $s$ and then
\begin{eqnarray*}
f(x) & = & (\sum\limits_jf_jm_j^*)(x) \\
 & = & \sum\limits_jm_j^*(x\cdot f_j) \;\;\;{(\rm by\,the\,left\,}C^*-{\rm module\,structure\,on\,}M^*) \\
 & = & \sum\limits_jm_j^*(\sum\limits_s(f_j(c_s)x_s)) \\
 & = & 0 \;\;\; ({\rm because\,}f_j(c_s)=0,\,\forall\,j,s)
\end{eqnarray*}
This shows that $f(x)=0$ and since $x$ is arbitrary, we get $f=0$.
\end{proof}

\subsection*{The Splitting property}

\begin{proposition}\label{Dom.p.Ideals}
If $C$ is colocal and $C^*$ has the (left) Dickson splitting property, then $C^*/I$ is semiartinian for any nonzero left ideal $I$ of $C^*$.
\end{proposition}
\begin{proof}
Let $I$ be a nonzero left ideal of $C^*$ and $0\neq f\in I$. Let $K=Af^\perp$; then since $Af$ is finitely generated, it is closed in the finite topology of $C^*$ (for example by \cite{I2}, Lemma 1.1) and we have $Af=K^\perp$. Note that $K\neq C$ since $Af\neq 0$. We have that $K$ is a left coideal of $C$, and as $K\neq C$, we can find a left coideal $Y\leq C$ such that $Y/K$ is finite dimensional nonzero (by the Fundamental Theorem of Comodules). Then there is an exact sequence of left $C^*$-modules
$$0\rightarrow (Y/K)^*\rightarrow Y^*\rightarrow K^*\rightarrow 0$$
Write $Y^*=\Sigma\oplus T$, with $\Sigma$ semiartinian and $T$ containing no semiartinian (or equivalently, no simple) submodules. Then $\Sigma\neq 0$ since $0\neq(Y/K)^*$ is a finite dimensional (thus semiartinian) left $C^*$-module contained in $Y^*$. But $Y^*$ is indecomposable by Lemma \ref{Dom.l.2}, and therefore $Y^*=\Sigma$ follows, i.e. $Y^*$ is semiartinian. The above sequence shows that $K^*$ is semiartinian too and since $K^\perp=Af\subseteq I$ we get an epimorphism $K^*\simeq C^*/K^\perp\rightarrow C^*/I\rightarrow 0$, and therefore $C^*/I$ is semiartinian.
\end{proof}

\begin{theorem}\label{T.1}
Let $C$ be a coalgebra such that $C^*$ has the Dickson splitting property for left modules. If $C^*$ is a domain, $C^*$ must be a finite dimensional division algebra (so it will have finite Loewy length).
\end{theorem}
\begin{proof}
Note that $C_0$ is finite dimensional by Corollary \ref{p.1}. We show that $C=C_0$, which will end the proof, since then $C^*=C_0^*$ is a finite dimensional semisimple algebra which is a domain, thus it must be a division algebra. \\
Assume $C_0\neq C$. Then $J=J(C^*)\neq 0$ and take $f\in J$, $f\neq 0$. Denote $M_n=(\overline{f}\circ\dots\circ\overline{f})^{-1}(C_n)=(\overline{f}^n)^{-1}(C_n)=(\overline{f^n})^{-1}(C_n)$, with $M_0={\rm Id}^{-1}(C_0)=C_0$. Note that $M_n\subseteq M_{n+1}$, $\overline{f}^n(M_n)=C_n$ since $\overline{f}$ is surjective by Lemma \ref{Dom.p.domain} and $M_n^\perp\neq 0$ (otherwise $M_n=(M_n^\perp)^{\perp}=0^\perp=C$ so $C_n=\overline{f^n}(M_n)=\overline{f}^n(C)=C$; but $C=C_n$ is excluded by assumption). \\
Denote $A=C^*$ and 
$$M=\frac{A}{M_0^\perp}\times \frac{A}{M_1^\perp}\times\dots\times\frac{A}{M_n^\perp}\times\dots=\prod\limits_{n\geq 0}\frac{A}{M_n^\perp}$$
as a left $A$-module. Also put $\lambda=(\hat\varepsilon,\hat{f},\hat{f^2},\dots)\in M$ (where $\hat h$ denotes the image of $h\in A$ modulo some $M_n^\perp$). Let $M=T\oplus X$ with $T$ semiartinian and $X$ containing no semiartinian modules. If $t_n=(\hat\varepsilon,\hat f,\dots,\hat{f^{n-1}},0,0,\dots)$ then $t_n\in \prod\limits_{0\leq i<n}A/M_i^\perp\,\times 0$ which is semiartinian since it is the quotient of $(A/M_{n}^\perp)^n$, and $A/M_n^\perp$ is semiartinian by Proposition \ref{Dom.p.Ideals}. Put $x_n=(0,0,\dots,0,\hat\varepsilon,\hat f,\hat{f^2},\dots,\hat{f^n},\dots)$ ($\varepsilon=1_A$ is on position "n", with positions starting from "0") and write $\lambda=t+x$ and $x_n=t_n'+x_n'$ with $t,t_n'\in T$ and $x,x_n'\in X$. Then
$$t+x=\lambda=t_n+f^n\cdot x_n= t_n+f^n(t_n'+x_n')=(t_n+f^n\cdot t_n')+f^n\cdot x_n'$$
shows that $x=f^n\cdot x_n'$ (since $M=T\oplus X$). Therefore if $x=(\hat{y_p})_{p}\in M$, $x_n'=(\hat{y_{n,p}})_p\in M$, $\hat{y_n}, \hat{y_{n,p}}\in A/M_n^\perp$, we get $\hat{y_p}=f^n\cdot \hat{y_{n,p}}$, for all $n,p$. Since $f\in J$, we get $\hat{y_p}\in J^n\cdot (A/M_p^\perp)$ for all $n$. Fixing $p$ and using Lemma \ref{Dom.l.3} we get that $\hat{y_p}=0$ (since $A/M_p^\perp\simeq M_p^*$). This holds for all $p$, hence $x=0$ and then $\lambda=t\in T$. Note that $\lambda\neq 0$ (since $\varepsilon\notin M_0^\perp=C_0^\perp=J\neq C$) and since $A\cdot \lambda$ is semiartinian, there is some $g\in A$ such that $Ag\lambda$ is a simple left $A$-module. Then $Ag\hat{f^n}=(Agf^n+M_n^\perp)/M_n^\perp\subseteq A/M_n^\perp$ is either simple or 0 for all $n$ (since it is a quotient of $Ag\lambda$) and so it must be annihilated by $J$ (use Remark \ref{r.J} for example). Thus $J\cdot g\hat{f^n}=\hat{0}$ in $A/M_n^\perp$ and so $J\cdot gf^n\subseteq M_n^\perp$ and then for $a\in J$ and $m\in M_n$ we have
\begin{eqnarray*}
0 & = & a\cdot g\cdot f^n(m)=a(m_1)gf^n(m_2)=a((gf^n)(m_2)m_1)\\
  & = & a(\overline{gf^n}(m))=a(\overline{g}\overline{f^n}(m))
\end{eqnarray*}
Therefore, $0=a(\overline{g}(\overline{f}^n(M_n)))=a(\overline{g}(C_n))$ for all $a\in J$, which shows that $\overline{g}(C_n)\subseteq J^\perp=C_0$. This holds for all $n$, showing that $\overline{g}(C)=g(\bigcup\limits_{n}C_n)\subseteq C_0$. But $\overline{g}\neq 0$ since $g\neq 0$, so $\overline{g}$ has to be surjective. But this is obviously a contradiction, because $\overline{g}(C)\subseteq C_0\neq C$, and the proof is finished. 
\end{proof}

The above also shows that a domain profinite algebra which has finite Loewy length (equivalently, $A=C^*$ with $C=C_n$ for some $n$) must necessarily be a division algebra. This can actually be easily proved directly by using Lemma \ref{Dom.p.domain}, as we invite the reader to note.


\section{Dickson's Conjecture for duals of coalgebras}

Denote by ${\mathcal T}$ the torsion preradical associated to the Dickson localizing subcategory of ${}_{C^*}{\mathcal M}$. If $A$ is an algebra such that $A/J$ is semisimple, we again use the observation that a left $A$-module $N$ is semisimple if and only if $JN=0$. Moreover, this implies that $N$ is semiartinian of finite Loewy length if and only if $J^nN=0$ for some $n\geq 0$.

\begin{proposition}\label{Dickson.p}
Let $C\neq 0$ be a colocal coalgebra such that $C^*$ has the left Dickson splitting property. Then ${\mathcal T}C^*\neq 0$. 
\end{proposition}
\begin{proof}
Assume otherwise. Then we show that $C^*$ is a domain, and then this will yield a contradiction by Theorem \ref{T.1}, since then $C^*$ is a finite dimensional division algebra so it is semiartinian. To see that $C^*$ is a domain, choose $0\neq f\in C^*$ and define $\varphi_f:C^*\rightarrow C^*$ by $\varphi_f(h)=hf$; then $\varphi_f$ is a morphism of left $C^*$-modules. If $\ker(\varphi_f)\neq 0$ then by Proposition \ref{Dom.p.Ideals}, $C^*\cdot f\simeq C^*/\ker(\varphi_f)$ is semiartinian. This shows that $0\neq C^*\cdot f\subseteq \Tt(C^*)$ which contradicts the assumption. Therefore $\ker(\varphi_f)=0$, following that $f$ is a non zero-divisor. This completes the proof.
\end{proof}

\begin{corollary}\label{4.local_splitting}
Let $C$ be a colocal coalgebra. If $C^*$ has the Dickson splitting property for left $C^*$-modules, then $C^*$ is left semiartinian.
\end{corollary}
\begin{proof}
We have $C^*={\mathcal T}C^*\oplus I$, for some left ideal $I$ of $C^*$. By the previous Proposition, ${\mathcal T}C^*\neq 0$, and since $C$ is colocal, $C^*$ is indecomposable by Lemma \ref{Dom.l.2}. Therefore we must have $I=0$ and $C^*={\mathcal T}C^*$ is left semiartinian. 
\end{proof}


\begin{theorem}
Let $A$ be a pseudocompact algebra, that is $A=C^*$ for a coalgebra $C$. If $A$ has the Dickson splitting property for left $A$-modules, then $A$ is left semiartinian.
\end{theorem}
\begin{proof}
Let $A=C^*$, $C$-coalgebra and $C=\bigoplus\limits_{i\in F}E_i$ be a decomposition of $C$ into left indecomposable injective comodules; then $F$ is finite by Corollary \ref{p.1}. We have $A=C^*=\bigoplus\limits_{i\in F}E_i^*$ with $E_i^*$ projective indecomposable left $A$-modules, so $E_i^*=Ae_i$ with $(e_i)_{i\in F}$ a complete system of indecomposable orthogonal idempotents. By \cite{NT}, Corollary 2.4 we have that each ring $e_iAe_i$ has the Dickson splitting property for left modules. By \cite{Rad}, Lemma 6 (also see \cite{CGT}), $eAe=eC^*e$ is also a pseudocompact algebra, dual to the coalgebra $eCe=\{e(c_1)c_2e(c_3)\mid c\in C\}$ with counit $e$ and well defined comultiplication $ece\mapsto ec_1e\otimes ec_2e$. Also, since $e_i$ are primitive, $e_iAe_i$ are local. Therefore, Corollary \ref{4.local_splitting} applies, and we get that $e_iAe_i$ are semiartinian for all $i\in F$. Now, by Theorem \ref{1.semiartinian} it follows that $A$ is semiartinian too (one can also use the fact that ${}_{e_iAe_i}{\mathcal M}$ are localizations of ${}_A{\mathcal M}$, and then apply \cite{CICN}, Proposition 3.5 1,(b)). 
\end{proof}

As an immediate consequence, we obtain the following result proved first in \cite{NT} and then independently in \cite{Cu} and \cite{I1}:

\begin{corollary}
Let $C$ be a coalgebra such that the rational submodule of any left $C^*$-module $M$ splits off in $M$. Then $C$ is finite dimensional.
\end{corollary}
\begin{proof}
Note that in this case $C^*$ has the Dickson splitting property for left modules: if $M$ is a left $C^*$-module, $M=R\oplus X$ with $R$ rational - thus semiartinian - and $Rat(X)=0$. Then $X$ contains no simple submodules since all simple modules are rational because $C^*/J(C^*)$ is finite dimensional semisimple in this case (see, for example, \cite{NT} or \cite[Proposition 1.2]{I1}). So $R$ is also the semiartinian part of $M$ and it is a direct summand of $M$. Thus it follows that $C^*$ is semiartinian from the previous Theorem. Now write $C^*=Rat(C^*)\oplus N$ with $Rat(N)=0$; then since $N$ is semiartinian we must have $N=0$ (otherwise $N$ contains simple rational submodules). Hence $Rat(C^*)=C^*$ and this module is also cyclic, so it is finite dimensional. Therefore $C$ is finite dimensional too.
\end{proof}


We note that in several particular cases, some more can be inferred in the case the Dickson splitting property holds (i.e. if $C^*$ is semiartinian).

\begin{lemma}\label{l.1}
If $J=J(A)$ is a finitely generated left ideal, where $A=C^*$ and the coalgebra $C$ is almost connected (i.e. the coradical $C_0$ is finite dimensional), then a finitely generated $A$-module is semiartinian if and only if it has finite Loewy length. 
\end{lemma}
\begin{proof}
It is enough to consider $M=Ax$ for some $x\in M$. If the Loewy length of $M$ is strictly greater than $\omega$, the first infinite ordinal, then there is $f\in A$ such that $fx\in \Ll_{\omega+1}(M)\setminus\Ll_\omega(M)$ so $fx+\Ll_\omega(M)/\Ll_\omega(M)$ is semisimple and then it is annihilated by $J$, (since $A/J$ is finite dimensional semisimple). Then $Jfx\subseteq\Ll_\omega(M)=\bigcup\limits_{n}\Ll_n(M)$. Let $g_1,\dots,g_s$ generate $J$ on the left; then one has $g_i\cdot fx\in\Ll_\omega(M)$ and therefore $g_i\cdot fx\in\Ll_{n_i}(M)$ for some $n_i$. This shows that $Jfx\subseteq \Ll_n(M)$ with $n={\rm max}\{n_1,\dots,n_s\}$ (as $J$ is generated on the left by the $g_i$'s). Then, since $J^n$ cancels $\Ll_n(M)$, we have that $J^{n+1}fx=J^n\cdot Jfx\subseteq J^n\Ll_n(M)=0$. Therefore, we get that $fx\in\Ll_{n+1}(M)\subseteq\Ll_\omega(M)$, a contradiction. \\
Now, note that we have $Ax=M=\Ll_\omega(M)=\bigcup\limits_{n}\Ll_n(M)$ and therefore $x\in\Ll_n(M)$ for some $n$, so $Ax\subseteq \Ll_n(M)$, and the proof is finished.
\end{proof}

\begin{corollary}
Let $C$ be a coalgebra such that $C^*$ is left semiartinian; if $J$ is finitely generated to the left, then $C$ has finite coradical filtration.
\end{corollary}

\begin{proposition}
If $C$ is an almost connected coalgebra with finite coradical filtration and such that the Jacobson radical $J$ of $C^*$ is finitely generated (on the left), then $C$ is finite dimensional.
\end{proposition}
\begin{proof}
Let $\{f_1,\dots,f_k\}$ be a set of generators of $J$. If $M$ is a finitely generated left $C^*$-module, say by $m_1,\dots,m_s$, then $JM$ is also finitely generated, by $\{f_im_j\}$. Indeed if $a\in J$, $m\in M$ then $m=\sum\limits_{j=1}^sa_jm_j$, $a_j\in A$ and since  $aa_j\in J$ we get $aa_j=\sum\limits_{i=1}^kb_{ij}f_i$. Therefore $am=\sum\limits_{j=1}^saa_jm_j=\sum\limits_{i=1}^k\sum\limits_{j=1}^sb_{ij}(f_im_j)$. Since $JM$ is generated by the elements of the form $am$, $a\in J$, $m\in M$, the claim follows. Proceeding inductively, this shows that $J^n$ is finitely generated, for all $n$. Then $J^n/J^{n+1}$ is finitely generated and semisimple (because $A/J$ is semisimple), thus it is finite dimensional. Therefore, inductively it follows that $A/J^n$ is finite dimensional. Finally, since $C_n=C$ for some $n$, $J^n=0$ so $A=C^*$ is finite dimensional.
\end{proof}

\begin{corollary}
If $A$ is a pseudo-compact algebra which is left semiartinian and $J(A)$ is finitely generated to the left (for example if $A$ is left noetherian), then $A$ is finite dimensional.
\end{corollary}

\begin{remark}
Naturally, the fact that $C^*$ is left semiartinian or even semiartinian of finite Loewy length (i.e. $C=C_n$ for some $n$) does not imply the finite dimensionality of $C$ (thus the result is of a completely different nature of that in \cite{NT} and \cite{I1}). Indeed, consider the coalgebra $C$ with basis $\{g\}\cup\{x_i,i\in I\}$ for an infinite set $I$ and comultiplication given by $g\mapsto g\otimes g$ and $x_i\mapsto x_i\otimes g+g\otimes x_i$ and counit $\varepsilon(g)=1$, $\varepsilon(x_i)=0$. Then $C_0=<g>$ and $C_1=C$, but $C$ is infinite dimensional.
\end{remark}

\begin{remark}
Thus the "Dickson Splitting conjecture" holds for the class of pseudocompact (profinite) algebras, which is the same as the class of algebras that are the dual of some coalgebra. As seen from above, in some situations it even follows that the algebra $A(=C^*)$ has finite Loewy length: if the Jacobson radical is finitely generated or if the algebra is a domain (in fact, even more follows in each of these cases). Then the following question naturally arises: if $C$ is a coalgebra such that $C^*$ is left semiartinian, does it follow that $C^*$ has finite Loewy length, equivalently, does $C$ have finite coradical filtration? At the same time, one can ask the question of whether $C^*$ - left semiartinian also implies $C^*$ - right semiartinian.
\end{remark}





\bigskip\bigskip\bigskip\bigskip

\begin{center}
\sc Acknowledgment
\end{center}
The authors would like to thank the referee for a thorough report on the paper which brought very useful observations, as well as for his/her suggestions on Theorem \ref{T.1} providing an ideea that simplified the original proof of this theorem.


\vspace*{3mm} 
\begin{flushright}
\begin{minipage}{148mm}\sc\footnotesize

Miodrag Cristian Iovanov\\
University of Bucharest, Faculty of Mathematics, Str.
Academiei 14,
RO-70109, Bucharest, Romania \&\\
State University of New York - Buffalo, 244 Mathematics Building, Buffalo NY, 14260-2900, USA\\
{\it E--mail address}: {\tt
yovanov@gmail.com}\vspace*{3mm}

Constantin N\u ast\u asescu\\
University of Bucharest, Faculty of Mathematics, Str.
Academiei 14,
RO-70109, Bucharest, Romania\\
\& UAL, Almeria,  Spain\\
{\it E--mail address}: {\tt 
constant@ual.es}\vspace*{3mm}

Blass Torrecillas Jover\\
UAL, Almeria,  Spain\\
{\it E--mail address}: {\tt
btorreci@ual.es}\vspace*{3mm}

\end{minipage}
\end{flushright}

\begin{thebibliography}{J\c{S}}

\bibitem[A]{A}
E. Abe, \emph{Hopf Algebras}, Cambridge University Press, Cambridge, London, 1977; 277p.

\bibitem[AN]{AN}
T. Albu, C. N\u ast\u asescu, \emph{Relative Finiteness in Module Theory}, Monogr. Textbooks Pure Appl. Math., vol. 84, Dekker, New York 1984.

\bibitem[AF]{AF}
D. Anderson, K.Fuller, \emph{Rings and Categories of Modules}, Grad. Texts in Math., Springer, Berlin-Heidelberg-New York, 1974.

\bibitem[CGT]{CGT}
J. Cuadra, J. G${\rm \acute{o}}$mez-Torrecillas, \emph{Idempotents and Morita-Takeuchi Theory}, Comm. Algebra 30, Issue 5 (2002), 2405-2436.

\bibitem[Cu]{Cu}
J. Cuadra, \emph{When does the rational submodule split off?}, Ann. Univ. Ferrarra -Sez. VII- Sc. Mat. Vol. LI (2005), 291-298.

\bibitem[CICN]{CICN}
F. Casta${\rm \tilde{n}}$o Iglesias, N. Chifan, C. Nastasescu, \emph{Localization on certain Grothendieck Categories}, preprint.

\bibitem[CDN]{NT0}
F. Casta$\rm\tilde{n}$o Iglesias, S. D\u asc\u alescu, C. N\u ast\u asescu, \emph{Symmetric Coalgebras}, J. Algebra {\bf 279} (2004) 326-344.

\bibitem[C]{C}
J.H. Cozzens, \emph{Homological properties of the ring of differential polynomials}, Bull. Amer. Math. Soc. 76 (1970), 75--79.

\bibitem[D]{D}
S.E. Dickson, \emph{A torsion theory for abelian categories}, Trans. Amer. Math. Soc. 35 (1972), 317-324.

\bibitem[DNR]{DNR} 
S. D\u asc\u alescu, C. N\u ast\u asescu, \c S. Raianu, \emph{Hopf Algebras: an introduction}. Vol. 235. Pure and Applied Mathematics, Marcel Dekker, New York, 2001.

\bibitem[FK]{FK}
J. Fuelberth, J. Kuzmanovich, \emph{On the structure of splitting rings}, Comm. Alg. {\bf 3} (1975), 913-949.

\bibitem[FT]{FT}
J. Fuelberth, M.L. Teply, \emph{The singular submodule of a finitely generated module splits off}, Pacific J. Math. 40 (1972), 73--82.

\bibitem[G]{G}
K.R. Goodearl, \emph{Torsion and the splitting properties}, Mem. Amer. Soc. 124 (1972), Providence. 

\bibitem[I]{I}
M.C.Iovanov, \emph{Characterization of PF rings by the finite topology on duals of $R$ modules}, An. Univ. Bucure\c sti Mat. 52 (2003), no. 2, 189-200.

\bibitem[I0]{I0}
M.C.Iovanov, \emph{Co-Frobenius Coalgebras}, J. Algebra 303 (2006), no. 1, 146--153; \\eprint arXiv:math/0604251, 
http://xxx.lanl.gov/abs/math.QA/0604251.

\bibitem[I1]{I1}
M.C.Iovanov, \emph{The Splitting Problem for Coalgebras: A Direct Approach}, Applied Categorical Structures 14 (2006) - Categorical Methods in Hopf Algebras - no. 5-6, 599-604.

\bibitem[I2]{I2}
M.C. Iovanov, \emph{The f.g. Rat-Splitting for Coalgebras}, eprint arXiv:math/0612478.

\bibitem[K1]{K1}
I. Kaplansky, \emph{Modules over Dedekind rings and valuation rings}, Trans. Amer. Math. Soc. 72 (1952) 327-340.

\bibitem[K2]{K2}
I. Kaplansky, \emph{A characterization of Pr$\rm\ddot{u}$fer domains}, J. Indian Math. Soc. 24 (1960) 279-281.

\bibitem[McL]{McL} 
S. Mac Lane, \emph{Categories for the Working Matematician}, Second Edition, Springer-Verlag, New York, 1971.

\bibitem[Mc1]{Mc1}
S. Mac Lane, \emph{Duality for groups}, Bull. Am. Math. Soc. {\bf 56}, 485-516 (1950).

\bibitem[NT]{NT}
C. N\u ast\u asescu, B. Torrecillas, \emph{The splitting problem for coalgebras}, J. Algebra {\bf 281} (2004), 144-149.

\bibitem[NT1]{NT1}
J. Gomez Torrecillas, C. N\u ast\u asescu, \emph{Quasi-co-Frobenius coalgebras}, J. Algebra 174 (1995), 909-923.

\bibitem[NT2]{NT2}
J. Gomez Torrecillas, C. Manu, C. N\u ast\u asescu, \emph{Quasi-co-Frobenius coalgebras II}, Comm. Algebra Vol 31, No. 10, pp. 5169-5177, 2003.

\bibitem[Rad]{Rad}
D.R. Radford, \emph{On the Structure of Pointed Coalgebras}, J. Algebra 77 (1982), 1-14.

\bibitem[Rot]{Rot}
J. Rotman, \emph{A characterization of fields among integral domains}, An. Acad. Brasil Cienc. 32 (1960) 193-194.

\bibitem[T1]{T1}
M.L. Teply, \emph{The torsion submodule of a cyclic module splits off}, Canad. J. Math. XXIV (1972) 450-464.

\bibitem[T2]{T2}
M.L. Teply, \emph{A history of the progress on the singular splitting problem}, Universidad de Murcia, Departamento de ${\rm\acute{A}}$lgebra y Fundamentos, Murcia, 1984, 46pp.

\bibitem[T3]{T3}
M.L. Teply, \emph{Generalizations of the simple torsion class and the splitting properties}, Canad. J. Math. 27 (1975) 1056-1074.

\bibitem[S]{S}
M.E. Sweedler, \emph{Hopf Algebras}, W.A. Benjamin, Inc.: New York, 1969; 336p.

\end{thebibliography}
\end{document}